\documentclass[11pt,reqno]{amsart}
\usepackage{fullpage}
\usepackage{comment}
\usepackage{footmisc}
\usepackage{amssymb}
\usepackage{tikz-cd}
\usepackage{tikz}
\usepackage{amsthm}
\usepackage{lscape}
\usepackage{array}
\usepackage{colonequals}
\usepackage{graphicx}
\usepackage{caption}
\usepackage{subcaption}
\usepackage[pagebackref=true, colorlinks = true, allcolors=magenta]{hyperref}
\usepackage{cleveref} 
\usepackage{longtable}
\usepackage{booktabs}
\usepackage{float}
\usepackage{makecell}
\usepackage{array}
\usepackage{footnote}

\renewcommand*{\backrefalt}[4]{%
  \ifcase #1 %
    \relax
  \or
    $\uparrow$#2.%
  \else
    $\uparrow$#2.%
  \fi%
}

\usepackage{todonotes}
\makeatletter
\providecommand\@dotsep{5}
\renewcommand{\listoftodos}[1][\@todonotes@todolistname]{%
	\@starttoc{tdo}{#1}}
\makeatother

\DeclareMathOperator{\rk}{rk}

\DeclareMathOperator{\PSL}{PSL}
\DeclareMathOperator{\GL}{GL}

\DeclareMathOperator{\gon}{gon}

\newcommand{\DeltaGrp}{\Delta}

\DeclareMathOperator{\Pic}{Pic}

\theoremstyle{plain}
\theoremstyle{plain}
\theoremstyle{plain}
\theoremstyle{plain}
\theoremstyle{plain}
\theoremstyle{plain}

\newtheorem{theorem}{Theorem}[section]
\newtheorem*{theorem*}{Theorem}
\newtheorem{lemma}[theorem]{Lemma}

\newtheorem{proposition}[theorem]{Proposition}
\newtheorem*{proposition*}{Proposition}
\newtheorem{corollary}[theorem]{Corollary}
\newtheorem*{corollary*}{Corollary}

\theoremstyle{definition}
\newtheorem{remark}[theorem]{Remark}

\numberwithin{equation}{section}
\newcommand{\Phiinf}{\Phi^{\infty}}

\newcommand{\Q}{\mathbb{Q}}
\newcommand{\Z}{\mathbb{Z}}

\newcommand{\C}{\mathbb{C}}
\newcommand{\F}{\mathbb{F}}
\newcommand{\PP}{\mathbb{P}}

\newcommand{\Qbar}{\overline \Q}

\newcommand{\fp}{\mathfrak{p}}

\newcommand{\tor}{\mathrm{tors}}

\def\tg#1#2{\mathbb Z/#1\mathbb Z \times \mathbb Z /#2 \mathbb Z}
\newcommand{\diamondop}[1]{\langle #1 \rangle} 
\newcommand{\set}[1]{\left\lbrace #1 \right\rbrace}

\newcommand{\XMN}[2]{X_1(#1,#2)}
\newcommand{\XN}[1]{X_1(#1)}
\newcommand{\JMN}[2]{J_1(#1,#2)}
\newcommand{\JN}[1]{J_1(#1)}
\newcommand{\Kz}[1]{\Q(\zeta_{#1})}

\newcommand{\lmfdbec}[3]{\href{https://www.lmfdb.org/EllipticCurve/Q/#1/#2/#3}{#1.#2#3}}

\newcommand{\githubbare}[1]{\href{https://github.com/nt-lib/twist-parametrized/blob/main/#1}{\path{#1}}}

\title{Torsion groups of elliptic curves that appear infinitely often over septic, octic and nonic fields}
\author[Najman]{Filip Najman}
\address{Filip Najman, University of Zagreb, Faculty of Science, Department of Mathematics, Bijeni\v{c}ka Cesta 30, 10000 Zagreb, Croatia}
\email{\url{fnajman@math.hr}}
\urladdr{\url{https://web.math.pmf.unizg.hr/~fnajman/}}

\author[Varivoda]{Marin Varivoda}
\address{Marin Varivoda, University of Zagreb, Faculty of Science, Department of Mathematics, Bijeni\v{c}ka Cesta 30, 10000 Zagreb, Croatia}
\email{marin.varivoda@math.hr}

\date{}
\thanks{The authors were supported by the project ``Implementation of cutting-edge research and its application as part of the Scientific Center of Excellence for Quantum and Complex Systems, and Representations of Lie Algebras'', PK.1.1.02, financed by the  European Regional Development Fund, by the Croatian Science Foundation under the project no. IP-2022-10-5008, and by the Institutional grant of University of Zagreb, Faculty of Science IK IA 1.1.3. Impact4Math.}

\begin{document}
\begin{abstract}
    We determine the sets $\Phiinf(n)$ of abelian groups that appear as torsion groups of infinitely many elliptic curves, up to $\Qbar$-isomorphism, over number fields of degree $n=7,8$ and $9$. 

    The proof translates the problem into one about low-degree points on modular
curves $X_1(m,n)$.  We construct the infinite families using modular units, and eliminate the remaining candidates using finite-field gonality computations, covering arguments, and a specialization argument for $W^0_d$.  The most difficult case is $X_1(37)$ in degree $9$, where the
Jacobian has positive rank.  We handle this case by showing that $W^0_9(X_1(37)_{\F_2})$ contains no translate of the
positive-rank elliptic factor induced by the morphism
$X_1(37)\to X_0^+(37)$.

\end{abstract}
\maketitle

\section{Introduction}
For an elliptic curve $E$ over a number field $K$, the Mordell-Weil theorem implies that the group $E(K)$ of $K$-rational points on $E$ is a finitely generated group, and hence the torsion subgroup $E(K)_\tor$ is finite. The group $E(K)_\tor$ is either cyclic or is a product of two cyclic groups. Hence 
$$E(K)_\tor\simeq \tg{m}{n},$$
and following \cite{DerickxSutherland}, we will for simplicity denote $\tg{m}{n}$ by $(m,n)$.

Let $\Phi(d)$ be the set of possible torsion groups of elliptic curves over all number fields of degree $d$. By Merel's theorem \cite{merel}, the set $\Phi(d)$ is finite for all positive integers $d$. We denote by $\Phiinf(d)$ the subset of groups in $\Phi(d)$ that occur for infinitely many $\Qbar$-isomorphism classes of elliptic curves over degree $d$ number fields. Thus \(\Phi^\infty(d)\) is the non-sporadic, or geometric,
part of the degree-\(d\) torsion problem.

The determination of $\Phi(d)$ and $\Phi^\infty(d)$  is a central problem in the arithmetic of elliptic curves over number fields. Mazur notably determined in \cite{mazur77} 
$$\Phi(1) = \{(1, n) : n \in \{1, \dots, 10, 12\}\} \cup \{(2, 2n) : 1 \le n \le 4\}.$$
Kamienny \cite{kamienny92}, building on work of Kenku and Momose \cite{KM88}, determined 
$$\Phi^\infty(2) = \{(1, n) : n \in \{1, \dots, 10, 12, 13, 14, 15, 16, 18\}\} \cup \{(2, 2n) : 1 \le n \le 6\} \cup \{(3, 3), (3,6), (4,4)\}.$$
We have $\Phi(1)=\Phi^\infty(1)$ and $\Phi(2)=\Phi^\infty(2).$ Jeon, Kim and Schweizer \cite{JeonKimSchweizer04} determined $$\Phi^\infty(3) = \{(1, n) : n \in \{1, \dots, 16, 18, 20\}\} \cup \{(2, 2n) : 1 \le n \le 7\},$$
 and the first author \cite{najman16} showed that $\Phi^\infty(3)\subsetneq \Phi(3)$ (by showing that $(1,21)\in  \Phi(3)$). The set $\Phi(3)=\Phi^\infty(3) \cup \{(1,21)\}$ was later determined by Derickx, Etropolski, Morrow, van Hoeij and Zureick-Brown \cite{Deg3Class}. Jeon, Kim and Park \cite{JeonKimPark06} determined 
\begin{align*}
    \Phi^\infty(4) = &\{(1, n) : n \in \{1, \dots, 18, 20, 21, 22, 24\}\} \cup \{(2, 2n) : 1 \le n \le 9\}\\ &\cup \{(3, 3n) : 1 \le n \le 2\} \cup \{(4, 4), (4,8), (5, 5), (6, 6)\}.
\end{align*}
It was only recently proved by Derickx and the first author that $\Phi(4)=\Phi^\infty(4)$ \cite{DerickxNajman25}. The sets \begin{align*}
    \Phi^\infty(5) = &\{(1, n) : 1 \le n \le 25, n \neq 23\} \cup \{(2, 2n) : 1 \le n \le 8\},\\
    \Phi^\infty(6) = &\{(1, n) : 1 \le n \le 30, n \notin \{23, 25, 29\}\} \cup \{(2, 2n) : 1 \le n \le 10\}\\ &\cup \{(3, 3n) : 1 \le n \le 4\} \cup \{(4, 4), (4, 8), (6, 6)\}
\end{align*}
have been determined by Derickx and Sutherland \cite{DerickxSutherland}. 
At the time of writing, the full set \(\Phi(d)\) is not known for any \(d>4\). However, the sets $S(d)$ of possible primes dividing the orders of groups in $\Phi(d)$ are known for $d\leq 7$ by work of Derickx, Kamienny, Stein and Stoll \cite{DKSS}, and Khawaja for $d=8$ \cite{Khawaja24}.

In this paper we determine the next three cases, namely
\(\Phi^\infty(7)\), \(\Phi^\infty(8)\), and \(\Phi^\infty(9)\).

\begin{theorem}\label{mainthm}
$\ $
\begin{enumerate}
    \item \[
\Phiinf(7)=\set{(1,n):1\leq n\leq 30,\ n\notin\set{25,29}}
\cup \set{(2,2n):1\leq n\leq 10}.
\]

\item
\begin{align*}
\Phiinf(8)=&\set{(1,n): n\in\set{1,\ldots,28,30,32,36}}
\cup \set{(2,2n):1\leq n\leq 12}\\
&\cup \set{(3,3n):1\leq n\leq 4}
\cup \set{(4,4n):1\leq n\leq 3}
\cup \set{(5,5),(6,6)}.
\end{align*}

\item 
\[
\Phiinf(9)=\set{(1,n): n\in\set{1,\ldots,28,30,36}}
\cup \set{(2,2n):1\leq n\leq 12}.
\]

\end{enumerate}
\end{theorem}

The proof follows the strategy introduced by Derickx and Sutherland for degrees \(5\) and \(6\).  The basic bridge is that torsion groups \((m,n)\)
are detected by low-degree points on the modular curves \(X_1(m,n)\), defined
over \(\mathbb Q(\zeta_m)\).  Thus \((m,n)\in\Phi^\infty(d)\) forces
\(X_1(m,n)\) to have infinitely many points of degree \(d/\varphi(m)\) over
\(\mathbb Q(\zeta_m)\), and conversely a map of that degree to
\(\mathbb P^1\) produces infinitely many such torsion groups.

We first use Abramovich's gonality bound, rank-zero information for
\(J_1(m,n)\), and known results on cyclic torsion to produce finite candidate sets \[
    T_d\cup E_d,\qquad d=7,8,9.
\] The groups in \(T_d\) are the groups appearing in Theorem~\ref{mainthm}; the groups in \(E_d\) are the excess candidates to be eliminated.  For every group in \(T_d\), we produce infinitely many examples, usually by finding an explicit modular unit (functions whose divisor is supported only on cusps, see e.g. \cite[Chapter 2]{KubertLang}) of the required degree.

To eliminate the groups in \(E_d\), we
combine finite-field gonality computations, rank-zero arguments, and maps to lower-genus modular curvessee
Proposition~\ref{prop:candidate-lists} and Section  \ref{sec:excess}. Here we make use of the huge advances made by Derickx and Terao \cite{DerickxTerao} in computing class groups and gonalities of curves over finite fields, without which some of our computations would likely not be feasible.

The only case which is not eliminated by these methods, and by far the most difficult case, is \((1,37)\) in
degree \(9\).   The Jacobian $J_1(37)$ has positive rank over $\Q$, and the gonality of $X_1(37)$ over $\Q$ is $18$, so the general inequality $\delta(X/\Q)\geq \frac{1}{2}\gon_\Q(X)$ is exactly insufficient.  We handle this case by a specialization computation on $W^0_9(X_1(37))$ over $\F_2$ which shows that no translate of the positive-rank elliptic factor of \(J_1(37)\) induced by
the quotient \[X_1(37)\to X_0^+(37)\simeq 37.a1\] lies in $W^0_9 (X_1(37))$.

Most of the code that proves our claims was written in Magma \cite{magma}. All the code for the computations in this paper is available at: 
\begin{center}
    \url{https://github.com/marin-varivoda/torsion_inf/tree/main}
\end{center}

OpenAI's GPT-5.5 Pro was used to assist in producing Magma code verifying Proposition~\ref{prop:X37}, using mathematical ideas supplied by the authors.

\section*{Acknowledgements} We are grateful to Maarten Derickx for providing unpublished code related to \cite{DerickxSutherland} and for many helpful conversations.

\section{Background and previous results} \label{sec: background}
In this section we collect results that we will use to prove \Cref{mainthm}. As is well-known, $E(K)_\tor$ can contain $(n,n)$ as a subgroup only if $\Q(\zeta_n)\subset K$. It follows that for an odd degree number field $K$, we have $(m,n)\leq E(K)_\tor$ only if $m=1$ or $2$.

We recall the standard bridge between torsion groups and modular curves.  Let $m\mid n$.  
If $K\supseteq \Kz{m}$, then the $K$-rational noncuspidal points on $\XMN{m}{n}$ classify triples
\[
(E,P_m,P_n),
\]
where $E/K$ is an elliptic curve and $P_m,P_n\in E(K)$ generate a subgroup isomorphic to $(m,n)$.

We use the following lemma of Derickx and Sutherland.

\begin{lemma}[{\cite[Lemma 5.1]{DerickxSutherland}}]\label{lem:DS}
Let $d$ be a positive integer and $m\mid n$.
\begin{enumerate}
\item If $(m,n)\in \Phiinf(d)$, then $\XMN{m}{n}/\Kz{m}$ has infinitely many points of degree $d/\varphi(m)$.
\item If $\XMN{m}{n}/\Kz{m}$ admits a $\Kz{m}$-rational map to $\PP^1$ of degree $d/\varphi(m)$, then $(m,n)\in \Phiinf(d)$.
\end{enumerate}
\end{lemma}

Following \cite{VirayVogt24}, for a nice\footnote{smooth, projective, geometrically integral} curve $X$ over a field $K$, we define the \textit{minimum density degree.}\footnote{We follow the notation of \cite{CGPS22} here; note that this is denoted $\min\delta (X/K)$ in \cite{VirayVogt24}.} (over $K$) $ \delta(X/K)$ to be the smallest integer $n$ such that $X$ has infinitely many points of degree $n$ over $K$. 

By $\gon_K X$ we denote the \textit{gonality} of $X$ over $K$, which is defined to be the smallest degree of a morphism from $X$ to $\PP^1$ defined over $K.$ Recall that for a field extension $L/K$ we always have $\gon_L X\leq \gon_K X$. By $J(X)$ we denote the Jacobian of $X$. 

We collect below some results we will need. 

\begin{proposition}\label{prop: gon_den} \cite[Propositions 2.2 and 2.3]{DerickxSutherland}
    Let $X/K$ be a nice curve over a number field. Then 
    $$\delta(X/K)\leq \gon_K X\leq  2\delta(X/K),$$
    where the first inequality is an equality if $\rk J(X)(K)=0.$
\end{proposition}

We recall Abramovich's bound \cite[Theorem 0.1]{abramovich} (see \cite[Theorem 4.1]{DerickxNajman25} for the constant $\frac{325}{2^{15}}$), for a congruence subgroup $H$ we have
\begin{equation} \label{eq2}
    \gon_\C X_{H}>\frac{325}{2^{15}}[\PSL_2(\Z):H].
\end{equation}

Applying this to $H=\Gamma_1(m,n)$ and using the elementary lower bound for the index gives, for $m\geq 2$,
\begin{equation}\label{eq:crude-index}
\gon_\C \XMN{m}{n}>\frac{325}{2^{16}}\cdot \frac{6}{\pi^2}mn^2,
\end{equation}
which immediately shows that there are only finitely many groups we need to consider in each degree.

The following previously known results are used to keep the candidate search finite and small.

\begin{theorem}\label{thm:previous}
The following hold.
\begin{enumerate}
\item \cite[Theorem 3]{derickxVH} The curve $\XN{N}$ has infinitely many degree $7$ points over $\Q$ precisely for
\[
N\in \set{1,\ldots,30}\setminus\set{25,29},
\]
and infinitely many degree $8$ points over $\Q$ precisely for
\[
N\in \set{1,\ldots,28,30,32,36}.
\]
\item \cite{BarsKontogeorgisXarles13} For all $m\geq 7$, the curve $X(m)\simeq \XMN{m}{m}$ has only finitely many quadratic points over $\Kz{m}$.
\item \cite[Theorem 4.1]{DerickxSutherland} Let $m\mid n$. The rank of $\JMN{m}{n}(\Kz{m})$ is zero in the following ranges:
\[
\begin{array}{c|cccccc}
m&1&2&3&4&5&6\\ \hline
n\leq &36&42&30&24&20&30.
\end{array}
\]

\item 
\cite[Theorem 3.1 (2)]{Deg3Class} \(\operatorname{rk} J_1(n)(\Q)=0\)
if and only if
$$
\begin{aligned}
n\in
\{ &1,\ldots,36,\,
38,\ldots,42,\,
44,\ldots,52,\,
54,55,56,59,60,62,64,66,68,69,70,71,72,\\&75,76,78,
81,
84,87,90,94,98,100,108,110,119,120,132,140,150,168,180\}.
\end{aligned}
$$
\end{enumerate}
\end{theorem}

For a smooth projective curve $C$ over a field, let $W^0_d(C)\subseteq \Pic^d(C)$ denote the image of the $d$-th symmetric power $C^{(d)}$ under the Abel--Jacobi map. We will use the following theorem as stated in \cite[Theorem 4.2]{DebarreFahlaoui}.

\begin{theorem}\label{thm:DF}
Let $C/K$ be a nice curve over a number field.  If $C$ has infinitely many closed points of degree $d$ over $K$, then at least one of the following holds:
\begin{enumerate}
\item $C$ admits a $K$-rational map of degree $d$ to $\PP^1$;
\item $W^0_d(C)$ contains a translate of a positive-rank abelian subvariety of $\Pic^0(C)$.
\end{enumerate}
\end{theorem}

We will also need the following Lemma to be able to check Theorem \ref{thm:DF} (2) by specializing to a finite field. 
\begin{lemma}\cite[Lemma 4.8]{DHJO26} 
\label{lem:specialization-W}
Let $R$ be a discrete valuation ring with fraction field $K$ and residue field $k$.  Let $\mathcal C/R$ be a smooth projective curve with geometrically irreducible fibers and suppose that $\mathcal C(R)\neq\emptyset$.  If $W^0_d(\mathcal C_K)$ contains a translate of an abelian subvariety $A\subseteq \Pic^0(\mathcal C_K)$, then $W^0_d(\mathcal C_k)$ contains a translate of the special fiber of the Zariski closure of $A$.
\end{lemma}

\section{The finite candidate lists}\label{sec:candidates}

We now explain how the finite candidate lists are obtained.  Suppose $(m,n)\in \Phiinf(d)$.  The Weil pairing gives $\varphi(m)\mid d$.  We first keep only the pairs satisfying the crude inequality obtained by combining Proposition \ref{prop: gon_den} and \eqref{eq2}, 

\begin{equation}\label{eq:refined-sieve}
\frac{325}{2^{15}}[\PSL_2(\Z):\Gamma_1(m,n)]< \frac{2d}{\varphi(m)}.
\end{equation}
In the rank-zero ranges of \Cref{thm:previous}, we use $d/\varphi(m)$ instead of $2d/\varphi(m)$.  Finally, we apply the restrictions from \Cref{thm:previous}, the cyclic lists for $d=7,8$ and the finiteness of quadratic points on $X(m)$ for $m\geq 7$ to obtain the following result.

\begin{proposition}\label{prop:candidate-lists}
If $(m,n)\in \Phiinf(d)$ for $d\in\set{7,8,9}$, then $(m,n)\in T_d\cup E_d$, where
\begin{align*}
T_7={}&\set{(1,n):1\leq n\leq 30,\ n\notin\set{25,29}}
\cup \set{(2,2n):1\leq n\leq 10},\\
E_7={}&\set{(2,2n):11\leq n\leq 15},\\[3pt]
T_8={}&\set{(1,n):n\in\set{1,\ldots,28,30,32,36}}
\cup \set{(2,2n):1\leq n\leq 12}\\
&\cup \set{(3,3n):1\leq n\leq 4}
\cup \set{(4,4n):1\leq n\leq 3}
\cup \set{(5,5),(6,6)},\\
E_8={}&\set{(2,2n):13\leq n\leq 16}
\cup \set{(2,44),(2,46)}\\
&\cup \set{(3,15),(3,18),(4,16),(5,10),(6,12)},\\[3pt]
T_9={}&\set{(1,n):n\in\set{1,\ldots,28,30,36}}
\cup \set{(2,2n):1\leq n\leq 12},\\
E_9={}&\set{(1,n):n\in \set{29,31,32,33,34,35,37,38,\ldots,46,48,50,53,57,58,63}}\\
&\cup \set{(2,2n):13\leq n\leq 18}
\cup \set{(2,44),(2,46)}.
\end{align*}
\end{proposition}

\begin{proof}
This is the output of the sieve using \eqref{eq:refined-sieve} described above.  The only external inputs are those listed in \Cref{thm:previous}. The calculation is elementary and uses only the index formula for $\Gamma_1(m,n)$.
\end{proof}

The resulting candidates naturally partition into an expected list $T_d$, each of whose groups occurs infinitely often, and an excess list $E_d$, for which this is not the case. Which of the groups end up in $T_d$, and which in $E_d$ (and why) will become apparent in the next sections.

\section{Constructing the infinite families}\label{sec:inclusion}

In this section we prove that the expected groups in $T_d$ really occur infinitely often.

\begin{lemma}\label{lem:genus_bound}
    Let $X$ be a nice curve over a number field $K$ with a $K$-rational point, and $d$ an integer. If $d\ge \max\{1,2g\}$, then $X$ has infinitely many points of degree $d.$
\end{lemma}
\begin{proof}
Let \(P\in X(K)\). We first show that there is a nonconstant rational
function \(f\in K(X)\) whose polar divisor is exactly \(dP\).

For \(g=0\), this is obvious as there are functions with a unique pole of any prescribed positive order.

Now suppose \(g\ge 1\), and let \(d\ge 2g\). By Riemann--Roch, both divisors
\(dP\) and \((d-1)P\) are nonspecial, since their degrees are greater than
\(2g-2\). Hence
\[
\ell(dP)=d+1-g \quad
\text{ and } \quad
\ell((d-1)P)=d-g.
\]
Therefore
\(
\ell(dP)>\ell((d-1)P).
\)
So we may choose
\(
f\in L(dP)\setminus L((d-1)P).
\)
This means that \(f\) has no poles away from \(P\), and has a pole of exact
order \(d\) at \(P\). Thus
\(
(f)_\infty=dP.
\)
Hence \(f\) defines a morphism
\(
f:X\longrightarrow \mathbb P^1_K
\)
of degree \(d\).

It now follows from Hilbert's Irreducibility Theorem (see \cite[Chapter 9]{Serre1997}) that $X$ has infinitely many points of degree $d$.
\end{proof}

\begin{proposition}\label{prop:Td-included}
For $d\in\set{7,8,9}$, every group in $T_d$ belongs to $\Phiinf(d)$.
\end{proposition}

\begin{proof}
By Lemma~\ref{lem:DS}, it is enough to construct, for every $(m,n)\in T_d$, a function
\[
f\in \Kz{m}(\XMN{m}{n})
\]
of degree $d/\varphi(m)$.

For some of the smaller pairs one can avoid computation by using Lemma \ref{lem:genus_bound}. For the remaining cases, we do this using modular units.  Start with a generating set $f_1,\ldots,f_r$ of modular units on $\XMN{m}{n}$.  Taking valuations at the cusps embeds the group of modular units modulo constants into a lattice:
\[
f_i\longmapsto \bigl(v_{C_1}(f_i),\ldots,v_{C_c}(f_i)\bigr)\in \Z^c.
\]
If $f=f_1^{e_1}\cdots f_r^{e_r}$, then the degree of $f$ is one half of the $\ell^1$-norm of the corresponding valuation vector.  Thus finding a modular unit of prescribed degree is a short-vector search in this lattice.

We implemented this search using Maarten Derickx's \texttt{mdsage} package.  In the notation of that package, the computation for $\XMN{m}{n}$ is reduced to the command
\[
\texttt{has\_modular\_unit\_of\_degree(Gamma11(m,n/m), d/euler\_phi(m))}.
\]
For every pair $(m,n)\in T_d$ and $d\in\set{7,8,9}$, the search returns a modular unit of degree $d/\varphi(m)$.  Hilbert irreducibility and Lemma~\ref{lem:DS} therefore imply $(m,n)\in \Phiinf(d)$.

\end{proof}

\section{Eliminating the excess candidates}\label{sec:excess}

It remains to prove that no element of $E_7$, $E_8$ or $E_9$ occurs infinitely often. We use the following elementary specialization principle throughout: if a curve $X/K$ has good reduction at a prime $\fp$ of $K$, then
\begin{equation}\label{eq:specialization-gonality}
\gon_K(X)\geq \gon_{\F_\fp}(X_\fp).
\end{equation}
Indeed, a function on $X$ of degree $n$ specializes to a function of degree at most $n$ on the reduction.

\subsection{Rank-zero candidates and finite-field gonalities}

The following table lists the finite-field gonality computations used in the rank-zero cases.  The column $p$ denotes the rational prime below the chosen prime $\fp$ of the field of definition.  The column $\gamma$ is the computed value or lower bound for $\gon_{\F_\fp}(X_\fp)$. Applying \eqref{eq:specialization-gonality} we get $\gon_{\Q(\zeta_m)}X_1(m,n)\geq \gamma$.

\begin{longtable}{@{}cccc|@{\quad}cccc@{}}
\caption{Gonality lower bounds}
\label{tab1}\\
\toprule
Curve & $p$ & $\gamma$ & time (h)
&
Curve & $p$ & $\gamma$ & time (h)\\
\midrule
\endfirsthead

\toprule
Curve & $p$ & $\gamma$ & time (h)
&
Curve & $p$ & $\gamma$ & time (h)\\
\midrule
\endhead

$\XMN{2}{22}$ & $3$  & $8$        & $0.009$
&
$\XMN{2}{32}$ & $3$  & $\geq 10$  & $0.340$\\

$\XMN{2}{24}$ & $7$  & $8$        & $0.170$
&
$\XMN{2}{34}$ & $3$  & $\geq 10$  & $0.790$\\

$\XMN{2}{26}$ & $3$  & $12$       & $0.384$
&
$\XMN{2}{36}$ & $5$  & $\geq 10$  & $1.888$\\

$\XMN{2}{28}$ & $3$  & $12$       & $0.325$
&
$\XN{41}$     & $2$  & $\geq 10$  & $0.364$\\

$\XMN{3}{15}$ & $2$  & $6$        & $0.001$
&
$\XN{42}$     & $5$  & $\geq 10$  & $0.981$\\

$\XMN{3}{18}$ & $7$  & $6$        & $0.004$
&
$\XN{44}$     & $3$  & $\geq 10$  & $0.322$\\

$\XMN{4}{16}$ & $5$  & $8$        & $0.006$
&
$\XN{45}$     & $2$  & $\geq 10$  & $0.336$\\

$\XMN{5}{10}$ & $11$ & $3$        & $<0.001$
&
$\XN{46}$     & $3$  & $\geq 10$  & $0.831$\\

$\XMN{6}{12}$ & $7$  & $6$        & $0.009$
&
$\XN{48}$     & $5$  & $\geq 10$  & $1.160$\\

$\XMN{2}{30}$ & $7$  & $\geq 10$  & $3.902$
&
$\XN{50}$     & $3$  & $\geq 10$  & $0.932$\\

\bottomrule
\end{longtable}

All curves in the table have rank-zero Jacobian over their field of definition.  Consequently, Proposition~\ref{prop: gon_den} and \eqref{eq:specialization-gonality} imply that the density degree is at least $\gamma$. This allows us to eliminate most of the groups in the sets of excess candidates $E_d$.

\begin{corollary}\label{cor:table-eliminations}
\begin{enumerate}
\item $E_7\cap \Phiinf(7)=\emptyset$
\item $E_8\cap \Phiinf(8)\subseteq \{(2,44),(2,46)\}$;
\item None of the groups in \( 
\set{(2,2n):13\leq n\leq 18}
\cup \set{(1,N):N\in\set{41,42,44,45,46,48,50}}\) are in $\Phiinf(9).$
\end{enumerate}
\end{corollary}

\begin{proof}
For a candidate $(m,n)$ in degree $d$, Lemma~\ref{lem:DS} would require infinitely many points of degree $d/\varphi(m)$ on $\XMN{m}{n}/\Kz{m}$.  In each case listed above, Table \ref{tab1} gives
\[
\delta\bigl(\XMN{m}{n}/\Kz{m}\bigr)>\frac{d}{\varphi(m)}.
\]
This contradiction eliminates the candidate.

For the remaining cyclic candidates with $N\leq 40$, except $N=32$ and $N=37$, we use the known gonalities of Derickx and van Hoeij \cite[Table 1]{derickxVH}.  
\[
\begin{array}{c|cccccccc}
N&29&31&33&34&35&38&39&40\\ \hline
\gon_\Q \XN{N}&11&12&10&10&12&12&14&12.
\end{array}
\]
Since $\rk J_1(n)(\Q)=0$ by Theorem \ref{thm:previous} (4), this rules out infinitely many degree $9$ points for all the $X_1(N)$ listed above. 

\end{proof}

The only cyclic candidates with $N\leq 40$ that remain after this step are $N=32$ and $N=37$ in degree $9$.

\begin{remark}
    For some cases we can avoid computations by the Castelnuovo-Severi inequality (see e.g. \cite[Theorem 3.11.3]{Stichtenoth09}). Let $f$ be the forgetful degree 2 map $f:X_1(2,30)\rightarrow X_1(30)$ and suppose $h:X_1(2,30)\rightarrow \PP^1$ is a map of degree $d$. As we have $g(X_1(2,30))=25$, and $g(X_1(30))=9.$ The Castelnuovo-Severi inequality now implies that either 
    $$25=g(X_1(2,30))\leq \deg f \cdot g(X_1(30)) + d \cdot g(\PP^1)+(d-1)(\deg f -1),$$
    implying $d \geq 8$, or $h$ has to factor through $X_1(30).$ As $\gon_\Q X_1(30)\geq 5$ (see \cite{DerickxSutherland}), we obtain in any case $\gon_\Q X_1(2,30)\geq 8,$ eliminating $(2,30)$ in degree $7$.
\end{remark}

\begin{remark}
    We note that some parts of Corollary \ref{cor:table-eliminations} ($(2,2n)$ for $n=11,13,15$ in degree 7) have also been proved in \cite[Corollary 4.21 (1)]{noordman2017gonality} using different methods. 
\end{remark}

\subsection{The maps $\XMN{2}{44}\to \XN{44}$, $\XMN{2}{46}\to \XN{46}$ and $\XN{58}\to \XN{29}$}

If $X\to Y$ is a finite map over a number field $K$, then infinitely many points of degree $d$ on $X$ would give infinitely many points of degree at most $d$ on $Y$.  Hence finiteness of degree at most $d$ points on $Y$ implies the corresponding finiteness statement on $X$.

The natural maps
\[
\XMN{2}{44}\longrightarrow \XN{44},\qquad
\XMN{2}{46}\longrightarrow \XN{46}
\]
eliminate $(2,44)$ and $(2,46)$ in degrees $8$ and $9$, because $\XN{44}$ and $\XN{46}$ have rank-zero Jacobians and finite-field gonality at least $10$ by the table above.

Similarly, the natural map
\[
\XN{58}\longrightarrow \XN{29}
\]
eliminates $(1,58)$ from the degree $9$ list, because $\XN{29}$ has rank-zero Jacobian and $\gon_\Q\XN{29}=11$.

\subsection{The cases $(1,43)$, $(1,53)$ and $(1,57)$}

We next eliminate the degree $9$ candidates $(1,43)$, $(1,53)$ and $(1,57)$.

\begin{proposition}\label{prop:cover-search}
For $N\in\set{43,53,57}$,
\[
\gon_{\F_2}\XN{N}\geq 19.
\]
Consequently, $\XN{N}/\Q$ has only finitely many points of degree $9$.
\end{proposition}

\begin{proof}
The second assertion follows from \eqref{eq:specialization-gonality} and Proposition~\ref{prop: gon_den}; if $\gon_\Q\XN{N}\geq 19$, then $\delta(\XN{N}/\Q)\geq 10$.

We describe the computation proving the finite-field lower bound.  Suppose first that $X=\XN{43}/\F_2$ admitted a map $f:X\to \PP^1$ of degree at most $18$.  The curve $X$ has $21$ rational places over $\F_2$, while $\PP^1(\F_2)$ has $3$ points.  Hence at least $7$ rational places must lie in a single fiber of $f$.  After composing with an automorphism of $\PP^1$, this fiber may be taken to be the fiber above $\infty$.  Thus the pole divisor of $f$ is an effective divisor of degree at most $18$ whose support contains at least $7$ rational places.

A direct enumeration of all such divisors is too large.  Instead, we construct a covering set $\mathcal T$ of effective divisors with the property that for every effective divisor $D$ of degree $18$ whose support contains at least $7$ rational places, there are $T\in \mathcal T$ and a diamond operator $\diamondop{a}$ such that
\[
\diamondop{a}D\leq T.
\]
It follows that $L(\diamondop{a}D)\subseteq L(T)$.  We then compute each Riemann--Roch space $L(T)$ and check that it contains no function of degree at most $18$.

The same method applies to $N=53$ and $N=57$. The covering-set sizes and runtimes were as follows.
\[
\begin{array}{c|c|c|c}
\text{Curve}&\text{genus}&|\mathcal T|&\text{time}\\ \hline
\XN{43}&57&971&5.53\text{ h}\\
\XN{53}&92&329&55.81\text{ h}\\
\XN{57}&85&308&43.58\text{ h}.
\end{array}
\]
In all three cases the search found no function of degree at most $18$, proving the desired lower bound.
\end{proof}

\subsection{The group $(1,63)$}

The curve $\XN{63}$ is handled by mapping to a lower-genus intermediate modular curve.  Let
\[
\DeltaGrp=\set{\pm 1,\pm 20,\pm 22}\leq (\Z/63\Z)^\times.
\]
There is a natural sequence of morphisms
\[
\XN{63}\longrightarrow X_{\DeltaGrp}(63)\longrightarrow X_0(63).
\]
The curve $X_{\DeltaGrp}(63)$ has genus $25$ and rank-zero Jacobian over $\Q$.  Using Zywina's Modular Magma package (see \cite{Zywina25}), a model can be obtained from the congruence subgroup of $\GL_2(\Z/63\Z)$ generated by
\[
\begin{pmatrix}1&1\\0&19\end{pmatrix},\qquad
\begin{pmatrix}20&30\\0&1\end{pmatrix},\qquad
\begin{pmatrix}43&54\\0&2\end{pmatrix}.
\]
A gonality computation using the algorithm from \cite{DerickxTerao} gives
\[
\gon_{\F_2} X_{\DeltaGrp}(63)=12.
\]
Since the Jacobian has rank zero over $\Q$, $X_{\DeltaGrp}(63)$ has only finitely many points of degree at most $11$ over $\Q$.  Therefore $\XN{63}$ has only finitely many degree $9$ points.  This eliminates $(1,63)$.

\subsection{The special case of $\XN{32}$}

The curve $\XN{32}$ requires a separate argument.  Its Jacobian has rank 0 over $\Q$ and there is a degree $8$ modular unit on $\XN{32}$, so $\XN{32}$ has infinitely many points of degree $8$.  However, we need to rule out infinitely many points of degree $9$.

We use the following criterion of Derickx and van Hoeij.

\begin{proposition}\label{prop:DvH-criterion}
Let $C/\Q$ be a smooth projective curve with $C(\Q)\neq\emptyset$ and good reduction at a prime $p$.  Let $W^r_d(K)$ denote the closed subscheme of $\Pic^d(C)(K)$ corresponding to line bundles $L$ of degree $d$ with
\[
\dim_K H^0(C,L)\geq r+1.
\]
Suppose that:
\begin{enumerate}
\item $J(C)(\Q)\to J(C)(\F_p)$ is injective;
\item $\F_p(C)$ contains no functions of degree $d$;
\item $W^2_d(\F_p)=\emptyset$;
\item $W^1_{d-i}(\Q)\to W^1_{d-i}(\F_p)$ is surjective for all $1\leq i\leq d-\gon_{\F_p}(C)$;
\item $C^{(i)}(\Q)\to C^{(i)}(\F_p)$ is surjective for all $1\leq i\leq d-\gon_{\F_p}(C)$.
\end{enumerate}
Then $\Q(C)$ contains no functions of degree $d$.
\end{proposition}

\begin{proof}
This is \cite[Proposition 7]{derickxVH}.
\end{proof}

\begin{lemma}\label{lem:X32}
The curve $\XN{32}/\Q$ has only finitely many points of degree $9$.
\end{lemma}

\begin{proof}
We apply Proposition~\ref{prop:DvH-criterion} to $C=\XN{32}$, $d=9$ and $p=3$. Assumption (1) in Proposition \ref{prop:DvH-criterion} is satisfied by \cite[Appendix]{katz81}. The computation over $\F_3$ proceeds as follows.

First, we construct a covering set $\mathcal T$ of $22606$ effective divisors.  This set covers all effective divisors of degree $9$ by effective divisors of degree at most $18$.  For each $T\in\mathcal T$, we compute the Riemann--Roch space $L(T)$ and search for functions of degree at most $7$ and of degree $9$.  No such functions are found.  Whenever a degree $8$ function is found, we save its pole divisor; this produces $56$ effective divisors occurring as pole divisors of degree $8$ functions.  For each of these $56$ divisors $D$, we verify that
\[
\dim_{\F_3}L(D)=2.
\]
We also verify that if an effective divisor dominates two distinct divisors from this list of $56$, then its degree is at least $12$.  Finally, each of the $56$ divisors is obtained by reducing a $\Q$-rational divisor supported on cusps.

These checks verify the hypotheses of Proposition~\ref{prop:DvH-criterion}; hence $\Q(\XN{32})$ contains no functions of degree $9$.  Since $\JN{32}(\Q)$ has rank zero, it follows from Theorem \ref{thm:DF} that infinitely many degree $9$ points would give a degree $9$ function.  Therefore only finitely many degree $9$ points exist on $\XN{32}$. 
\end{proof}

\subsection{The group $(1,37)$}

We now eliminate the last degree $9$ excess candidate.  This is the only case in which the rank-zero equality between density degree and gonality is unavailable and $\gon_\Q X_1 (m,n)$ is not larger than $18$. We prove this following the strategy of \cite[Proposition 4.9]{DHJO26}.
However, instead of using explicit Weierstrass coordinates for the smooth elliptic quotient, we will have to work with a singular plane model of
\(X_0^+(37)_{\mathbb F_2}\) and with the corresponding places on its normalization, which makes the necessary computations more involved.

\begin{proposition}\label{prop:X37}
The curve $\XN{37}/\Q$ has only finitely many points of degree $9$.
\end{proposition}

\begin{proof}
Derickx and van Hoeij proved that $\gon_\Q \XN{37}=18$ \cite{derickxVH}.  Hence there is no degree $9$ map $\XN{37}\to \PP^1$ over $\Q$.

It remains, by Theorem~\ref{thm:DF}, to rule out the possibility that $W^0_9(\XN{37})$ contains a translate of a positive-rank abelian subvariety of $J_1(37)$.  A modular-symbol computation of the newform decomposition of $\JN{37}$, or equivalently the corresponding LMFDB newform data \cite{LMFDB}, shows that the only positive-rank simple factor, up to isogeny, is the elliptic factor induced by the quotient map
\[
\pi:X_1(37)\longrightarrow X_0^+(37)\simeq  \lmfdbec{37}{a}{1}.
\]
which has Weierstrass equation
\[
y^2+y=x^3-x.
\]
Thus it suffices to rule out translates of this elliptic curve factor.

We use the LMFDB plane model of \(X_1(37)\), written in Tate normal
form.  Let \(\mathcal E\) denote the universal elliptic curve over
\(\mathbb F_2(X_1(37))\), and let \(P=(0,0)\) be the marked point of
order \(37\).  Put
\[
    j_1=j(\mathcal E),\qquad
    j_2=j(\mathcal E/\langle P\rangle).
\]
The Fricke involution on \(X_0(37)\) interchanges \(j_1\) and \(j_2\).
Hence
\[
    S_{37}=j_1+j_2,\qquad T_{37}=j_1j_2
\]
are functions on \(X_0^+(37)\).  
These functions satisfy the relation
\[
\begin{aligned}
q(S,T)={}&S^{38}+S^{32}T^5+S^{24}T^{13}+S^{14}T^{20}
        +S^8T^{29}+S^8T^{25}  \\
        &+S^2T^{34}+T^{37}+T^{33}=0 .
\end{aligned}
\]
Let \(Y\) be 
the normalization of the plane curve
\(q(S,T)=0\).  We verify that \(g(Y)=1\) and \(
    \#Y(\mathbb F_2)=5,
\)
and from our construction it follows that
\[
    Y\simeq X_0^+(37)_{\mathbb F_2}\simeq 37.a1_{\mathbb F_2}.
\]
The quotient map
\[
    \pi:X_1(37)_{\mathbb F_2}\longrightarrow Y
\]
is induced on function fields by
\[
    S\longmapsto S_{37},\qquad T\longmapsto T_{37}.
\]

The computation uses two rational places of \(Y\).  The first is the
place at infinity cut out by the local functions
\(
    \frac{1}{S}, \frac{T^{36}}{S^{37}},
\)
and the second is the branch above the singular plane point \(S=T=0\)
cut out on the normalization by
\(
    S, \frac{T}{S}+1.
\)
It then computes their pullbacks to
\(X_1(37)_{\mathbb F_2}\). We denote the corresponding pullback divisors on \(X_1(37)_{\mathbb F_2}\)
by \(F_\infty\) and \(F_0\), respectively.

Set
\[
    D_0=F_\infty-F_0\in \operatorname{Div}^0(X_1(37)_{\mathbb F_2}).
\]
and let \(C_0=[D_0]\in J_1(37)(\mathbb F_2)\) be its divisor class. 
The code verifies that
\(
    5C_0 = 0
\)
and that
\(
    i\cdot C_0\not\neq 0  \text{ for } i=1,2,3,4.
\)
Thus \(C_0\) has order \(5\) in
\(J_1(37)(\mathbb F_2)\).  

Since \(X_1(37)\) has an \(\mathbb F_2\)-rational cusp, every
\(\mathbb F_2\)-rational point of \(W^0_9(X_1(37)_{\mathbb F_2})\) is
represented by an effective \(\mathbb F_2\)-rational divisor of degree
\(9\).  We enumerate all such divisors \(D\) and for each divisor check whether the dimension of at least one of the Riemann--Roch spaces
\[
    L(D+D_0),\quad
    L(D+2D_0),\quad
    L(D+3D_0),\quad
    L(D+4D_0)
\]
is zero.  Equivalently, for every effective divisor \(D\) of degree
\(9\) on \(X_1(37)_{\mathbb F_2}\), there is some
\(i\in\{1,2,3,4\}\) such that
\[
    D+iD_0\notin W^0_9(X_1(37)_{\mathbb F_2}).
\]
Therefore no translate of the order-\(5\) subgroup coming from the
reduced elliptic quotient is contained in
\(W^0_9(X_1(37)_{\F_2})\).

By Lemma~\ref{lem:specialization-W}, if
\(W^0_9(X_1(37)_{\mathbb Q})\) contained a translate of the positive-rank
elliptic quotient, then after reduction modulo \(2\) the
special fiber \(W^0_9(X_1(37)_{\F_2})\) would contain the corresponding
translate of the positive-rank elliptic quotient.  The computation shows that this is not the case, so we conclude that \(W^0_9(X_1(37)_{\mathbb Q})\) contains no translate of the positive-rank elliptic factor of \(J_1(37)\) induced by
the quotient \(X_1(37)\to X_0^+(37)\simeq 37.a1\) and hence $\XN{37}$ has only finitely many degree $9$ points over $\Q$.

\end{proof}

\begin{remark}
    In the proof of Proposition \ref{prop:X37} the divisor enumeration was implemented using an amortized
Riemann--Roch computation, following the local-expansion method of
Derickx--Terao \cite{DerickxTerao}.  For each ``degree type" (by which we partition parts of the computation) we choose an effective divisor \(H\) dominating all divisors \(D\) of that type.  Writing \(\delta=\delta_+-\delta_-\), we set
\[
    G=H+\delta_+,\qquad R=H+\delta_- .
\]
Then, for every \(D\leq H\), we get
\(
    D+\delta = G-(R-D).
\)
Thus \(L(D+\delta)\) is obtained from the fixed space \(L(G)\) by imposing
vanishing conditions along \(R-D\).  These conditions are expressed using
local expansions at the closed points in \(\operatorname{Supp}(R)\), and
after expanding residue-field coefficients over the base field they become
linear equations over the finite field.  Hence each membership test
\(D+\delta\in W_d^0(X_1(37))\) reduces to a matrix-rank computation.

This is not just a coding convenience; it is necessary to make the comptuation feasible. Instead of computing millions of separate Riemann–Roch spaces, we compute one larger space and then recover the smaller Riemann-Roch subspaces by linear algebra.
\end{remark}

\section{Proof of Theorem \ref{mainthm}}

We now combine the previous sections.  By Proposition~\ref{prop:candidate-lists}, every group in $\Phiinf(d)$ for $d=7,8,9$ lies in $T_d\cup E_d$.  By Proposition~\ref{prop:Td-included}, every group in $T_d$ belongs to $\Phiinf(d)$.  Corollary~\ref{cor:table-eliminations} eliminates all of $E_7$ and all of $E_8$ except $(2,44)$ and $(2,46)$, which are eliminated by the maps to $\XN{44}$ and $\XN{46}$.  Hence $\Phiinf(7)=T_7$ and $\Phiinf(8)=T_8$.

For $d=9$, Corollary~\ref{cor:table-eliminations}, the known gonalities of $\XN{N}$ for $N\leq 40$, the maps to $\XN{29}$, $\XN{44}$ and $\XN{46}$, the cover search for $N=43,53,57$, the intermediate curve $X_{\DeltaGrp}(63)$, Lemma~\ref{lem:X32}, and Proposition~\ref{prop:X37} eliminate every element of $E_9$.  Hence $\Phiinf(9)=T_9$.  This proves Theorem \ref{mainthm}.

\begin{remark}
The proof for $(1,37)$ is the only place where positive rank enters the degree $9$ elimination.  The gonality computation gives only $\delta(\XN{37}/\Q)\geq 9$, so the specialization argument on $W^0_9(X_1(37))$ is needed to rule out infinite degree $9$ points.
\end{remark}

\begin{remark}
    All of our computations were performed on the Mordell server at the Department of Mathematics, University of Zagreb with a AMD Epyc 9175F CPU running at 4.20GHz and with 384 GB of RAM. The total wall-clock time for the computations used in this paper was
    slightly more than 115 hours.
\end{remark}

\bibliographystyle{alpha}
\bibliography{bibliography1}
\end{document}